\documentclass[11pt,letterpaper]{article}
\usepackage[utf8]{inputenc}
\usepackage[margin=1in]{geometry}
\usepackage{amsmath, amssymb, amsfonts}
\usepackage{graphicx}
\usepackage{amsthm}
\usepackage{enumitem}
\usepackage{bm}
\usepackage{color}
\usepackage{geometry}
\usepackage{complexity}
\usepackage{algorithm}
\usepackage{algpseudocode}
\usepackage{float}
\theoremstyle{plain}
\usepackage[hidelinks]{hyperref}

\newtheorem{theorem}{Theorem}[section]
\newtheorem{lemma}[theorem]{Lemma}
\newtheorem{claim}[theorem]{Claim}

\newtheorem{corollary}[theorem]{Corollary}

\theoremstyle{definition}
\newtheorem{example}[theorem]{Example}

\newtheorem{remark}[theorem]{Remark}

\makeatletter
\@addtoreset{equation}{section}

\makeatother

\title{Minimizing Submodular Functions over Hierarchical Families}
\author{Ryuhei Mizutani\thanks{Faculty of Science and Technology, Keio University, Kanagawa, 223-8522, Japan. E-mail: \texttt{mizutani@math.keio.ac.jp}}}

\begin{document}

\maketitle

\begin{abstract}
This paper considers submodular function minimization (SFM) restricted to a family of subsets.
We show that SFM over the complements of families with certain hierarchical structures can be solved in polynomial-time.
This yields a polynomial-time algorithm for SFM over the complements of various families, such as intersecting families, crossing families, and the unions of lattices.
Moreover, this tractability result partially settles the open question posed by N\"{a}gele, Sudakov, and Zenklusen on polynomial-solvability of SFM over the intersection of parity families.
Furthermore, our tractability result implies that for a constant positive integer $k$, the $k$-th smallest
value of a submodular function can be obtained in polynomial-time.
\end{abstract}

\section{Introduction}
\label{sec:intro}
Let $V$ be a finite set.
A set function $f:2^V\to \mathbb{Z}$ is called \textit{submodular} if $f(X)+f(Y)\ge f(X\cup Y)+f(X\cap Y)$ holds for all $X,Y\subseteq V$.
It is a central result in combinatorial optimization that submodular functions can be minimized in polynomial-time in the evaluation oracle model \cite{grotschel1981,iwata2001,schrijver2000}.

While submodular function minimization (SFM) is efficiently solvable, polynomial-solvability of SFM restricted to a set family is known only for a few classes of families.
Classical examples of such families are intersecting and crossing families (see, e.g., \cite[Volume B]{schrijver2003}).
As another tractable case, Goemans and Ramakrishnan \cite{goemans1995} showed that SFM can be solved in polynomial-time over parity families.
Their result generalizes polynomial-solvability of SFM over the families of all odd sets and all even sets, and more generally, over triple families \cite{grotschel1981}. 
N\"{a}gele, Sudakov, and Zenklusen \cite{nagele2019} pursued a different extension of the odd/even-cardinality constraint to the congruency constraint, showing that SFM over all sets of cardinality $r$ mod $m$, where $m$ is a constant prime power, can be solved in polynomial-time.
In this paper, we introduce a new class of families, which includes complements of various families such as intersecting families, crossing families, and the unions of lattices.
We show that SFM over such  families can be solved in polynomial-time, which partially settles the open question posed by N\"{a}gele, Sudakov, and Zenklusen \cite{nagele2019}.

To provide a new tractable class of families, we introduce some definitions.
A set family $\mathcal{F}\subseteq 2^V$ is called a \textit{lattice} if for all $X,Y\in \mathcal{F}$, we have $X\cup Y,X\cap Y\in \mathcal{F}$.
For a positive integer $k$, we call a family $\mathcal{H}\subseteq 2^V$ a \textit{$k$-hierarchical lattice} if there exists a chain $\mathcal{C}$ of families
\begin{align*}
\mathcal{C}: \emptyset=\mathcal{H}_0\subseteq \mathcal{H}_1\subseteq \cdots \subseteq \mathcal{H}_k=\mathcal{H}
\end{align*}
such that
\begin{align*}
(*)\qquad
&\mathrm{for\ all\ }i=1,\ldots,k\ \mathrm{and\ }X,Y\in \mathcal{H}_i\setminus \mathcal{H}_{i-1},\ \mathrm{we\ have\ that\ }X\cup Y,X\cap Y\in \mathcal{H}_i\setminus \mathcal{H}_{i-1},\ \\&\mathrm{or\ at\ least\ one\ of\ }X\cup Y, X\cap Y\ \mathrm{belongs\ to\ }\mathcal{H}_{i-1}.
\end{align*}
Note that if $k=1$, then a $k$-hierarchical lattice coincides with a lattice.
We provide several examples of $k$-hierarchical lattices.

\begin{example}[Intersecting family]
\label{ex1}
Let $\mathcal{F}\subseteq 2^V$ be an \textit{intersecting family}, i.e., for all $X,Y\in \mathcal{F}$ with $X\cap Y\ne \emptyset$, we have $X\cup Y,X\cap Y\in \mathcal{F}$.
Define $\mathcal{H}_1:=\{\emptyset\}$ and $\mathcal{H}_2:=\mathcal{F}\cup \{\emptyset\}$.
Then, $\mathcal{H}_1$ is a lattice, and for all $X,Y\in \mathcal{H}_2\setminus \mathcal{H}_1=\mathcal{F}\setminus \{\emptyset\}$, we have $X\cap Y=\emptyset\in \mathcal{H}_1$ or $X\cup Y,X\cap Y\in \mathcal{H}_2\setminus \mathcal{H}_1$.
Hence, $\mathcal{H}_2$ is a 2-hierarchical lattice.
\end{example}

\begin{example}[Crossing family]
\label{ex2}
Let $\mathcal{F}\subseteq 2^V$ be a \textit{crossing family}, i.e., for all $X,Y\in \mathcal{F}$ with $X\cap Y\ne \emptyset$ and $X\cup Y\ne V$, we have $X\cup Y,X\cap Y\in \mathcal{F}$.
Define $\mathcal{H}_1:=\{\emptyset,V\}$ and $\mathcal{H}_2:=\mathcal{F}\cup \{\emptyset,V\}$.
Then, $\mathcal{H}_1$ is a lattice, and for all $X,Y\in \mathcal{H}_2\setminus \mathcal{H}_1=\mathcal{F}\setminus \{\emptyset,V\}$, we have $X\cap Y=\emptyset\in \mathcal{H}_1$ or $X\cup Y=V\in \mathcal{H}_1$ or $X\cup Y,X\cap Y\in \mathcal{H}_2\setminus \mathcal{H}_1$.
Hence, $\mathcal{H}_2$ is a 2-hierarchical lattice.
\end{example}

\begin{example}[Union of $k$ lattices]
\label{ex3}
Let $\mathcal{L}_1,\mathcal{L}_2,\ldots,\mathcal{L}_k\subseteq 2^V$ be lattices.
Define $\mathcal{H}_i:=\mathcal{L}_1\cup \mathcal{L}_2\cup \cdots \cup \mathcal{L}_i$ for $i=1,2,\ldots,k$.
Then, $\mathcal{H}_1=\mathcal{L}_1$ is a lattice.
For all $X,Y\in \mathcal{H}_i\setminus \mathcal{H}_{i-1}=\mathcal{L}_i\setminus \mathcal{H}_{i-1}$ and $i=2,\ldots,k$, we have $X\cup Y,X\cap Y\in \mathcal{L}_i$ since $\mathcal{L}_i$ is a lattice.
This implies that $X\cup Y,X\cap Y\in \mathcal{L}_i\setminus \mathcal{H}_{i-1}$ or at least one of $X\cup Y,X\cap Y$ belongs to $\mathcal{H}_{i-1}$.
Hence, $\mathcal{H}_k$ is a $k$-hierarchical lattice.
\end{example}

It is well known that the minimizers of a submodular function form a lattice (see, e.g., \cite{fujishige2005}).
The following example shows that the family of subsets whose function values are at most the $k$-th smallest value of a submodular function is a $k$-hierarchical lattice.

\begin{example}[$k$-th smallest value of a submodular function]
\label{ex4}
Let $f:2^V\to \mathbb{Z}$ be a submodular function.
For $i=1,2,\ldots,k$, define $\mathcal{H}_i$ as the family of subsets $X\subseteq V$ whose function values $f(X)$ are at most the $i$-th smallest value of $f$.
Then, since the minimizers of $f$ form a lattice, $\mathcal{H}_1$ is a lattice.
For $i=2,\ldots,k$, take $X,Y\in \mathcal{H}_i\setminus \mathcal{H}_{i-1}$.
Then, the submodular inequality $f(X)+f(Y)\ge f(X\cup Y)+f(X\cap Y)$ implies that if $X\cup Y\notin \mathcal{H}_i\setminus \mathcal{H}_{i-1}$ or $X\cap Y\notin \mathcal{H}_i\setminus \mathcal{H}_{i-1}$, then at least one of $X\cup Y$ and $X\cap Y$ belongs to $\mathcal{H}_{i-1}$.
Hence, $\mathcal{H}_{k}$ is a $k$-hierarchical lattice.
\end{example}

In this paper, we consider SFM over the complement of a $k$-hierarchical lattice, that is, the problem of minimizing a submodular function over sets outside a $k$-hierarchical lattice.
We first establish structural theorems showing that a minimizer of a submodular function over the complement of a $k$-hierarchical lattice is also a minimizer over a certain lattice.
Then, we derive a polynomial-time algorithm for SFM over the complement of a $k$-hierarchical lattice from these structural theorems.
This yields polynomial-time algorithms for SFM over the complements of several families, such as an intersecting family, a crossing family, the union of lattices, and the family of sets whose function values are at most the $k$-th smallest value of a submodular function.

The rest of this paper is organized as follows.
Section \ref{sec:result} provides structural theorems on the minimizers of a submodular function over the complement of a $k$-hierarchical lattice, which lead to a polynomial-time algorithm for SFM over that family.
Section \ref{sec:proof} provides proofs of the results in Section \ref{sec:result}.
Section \ref{sec:app} gives applications of the polynomial-time algorithm in Section \ref{sec:result} to SFM over several families.

\section{Main Results}
\label{sec:result}
We call a set family $\mathcal{F}\subseteq 2^V$ the \textit{complement of a $k$-hierarchical lattice} if $2^V\setminus \mathcal{F}$ is a $k$-hierarchical lattice.
For disjoint subsets $S,T\subseteq V$, we define
\begin{align*}
[S,T]:=\{X\subseteq V\mid S\subseteq X\subseteq V\setminus T\}.
\end{align*}
Note that $[S,T]$ is a lattice for every disjoint subsets $S,T\subseteq V$.
Our main result is the following structural theorem on the minimizers of a submodular function over the complement of a $k$-hierarchical lattice.
\begin{theorem}
\label{thm:minimizer}
For a positive integer $k$, let $\mathcal{F}\subseteq 2^V$ be the complement of a $k$-hierarchical lattice.
Let $f:2^V\to \mathbb{Z}$ be a submodular function.
Let $X^*\in \mathcal{F}$ be a minimizer of $f$ over $\mathcal{F}$.
Then, there exist $S,T\subseteq V$ with $S\cap T=\emptyset$ and $\max\{|S|,|T|\}\le k$ such that $X^*$ is a minimizer of $f$ over the lattice $[S,T]$.
\end{theorem}

To obtain a polynomial-time algorithm for SFM over the complement of a $k$-hierarchical lattice, we show the following  slight refinement of Theorem \ref{thm:minimizer}.

\begin{theorem}
\label{thm:minimal_minimizer}
For a positive integer $k$, let $\mathcal{F}\subseteq 2^V$ be the complement of a $k$-hierarchical lattice.
Let $f:2^V\to \mathbb{Z}$ be a submodular function.
Let $X^*\in \mathcal{F}$ be a minimal minimizer of $f$ over $\mathcal{F}$.
Then, there exist $S,T\subseteq V$ with $S\cap T=\emptyset$ and $\max\{|S|,|T|\}\le k$ such that $X^*$ is the unique minimal minimizer of $f$ over the lattice $[S,T]$.
\end{theorem}

By Theorem \ref{thm:minimal_minimizer}, we can obtain a minimizer over the complement of a $k$-hierarchical lattice by solving SFM over $[S,T]$ for all $S,T\subseteq V$ with $S\cap T=\emptyset$ and $\max\{|S|,|T|\}\le k$.
This yields the following result, given a membership oracle for the complement of a $k$-hierarchical lattice.

\begin{theorem}
\label{thm:algorithm}
For a positive integer $k$, let $\mathcal{F}\subseteq 2^V$ be the complement of a $k$-hierarchical lattice.
Let $f:2^V\to \mathbb{Z}$ be a submodular function.
Suppose that we are given an oracle for testing whether $X\in \mathcal{F}$ or not for any given $X\subseteq V$.
Then, a minimizer of $f$ over $\mathcal{F}$ can be obtained in $O(n^{2k}(\tau(n)+\gamma))$ time, where $n:=|V|$, $\tau(n)$ denotes the time required for minimizing a submodular function on a ground set of size $n$, and $\gamma$ denotes the time required for one membership oracle query for $\mathcal{F}$.
\end{theorem}

\begin{figure}[tb]
  \centering
  \includegraphics[width=15cm]{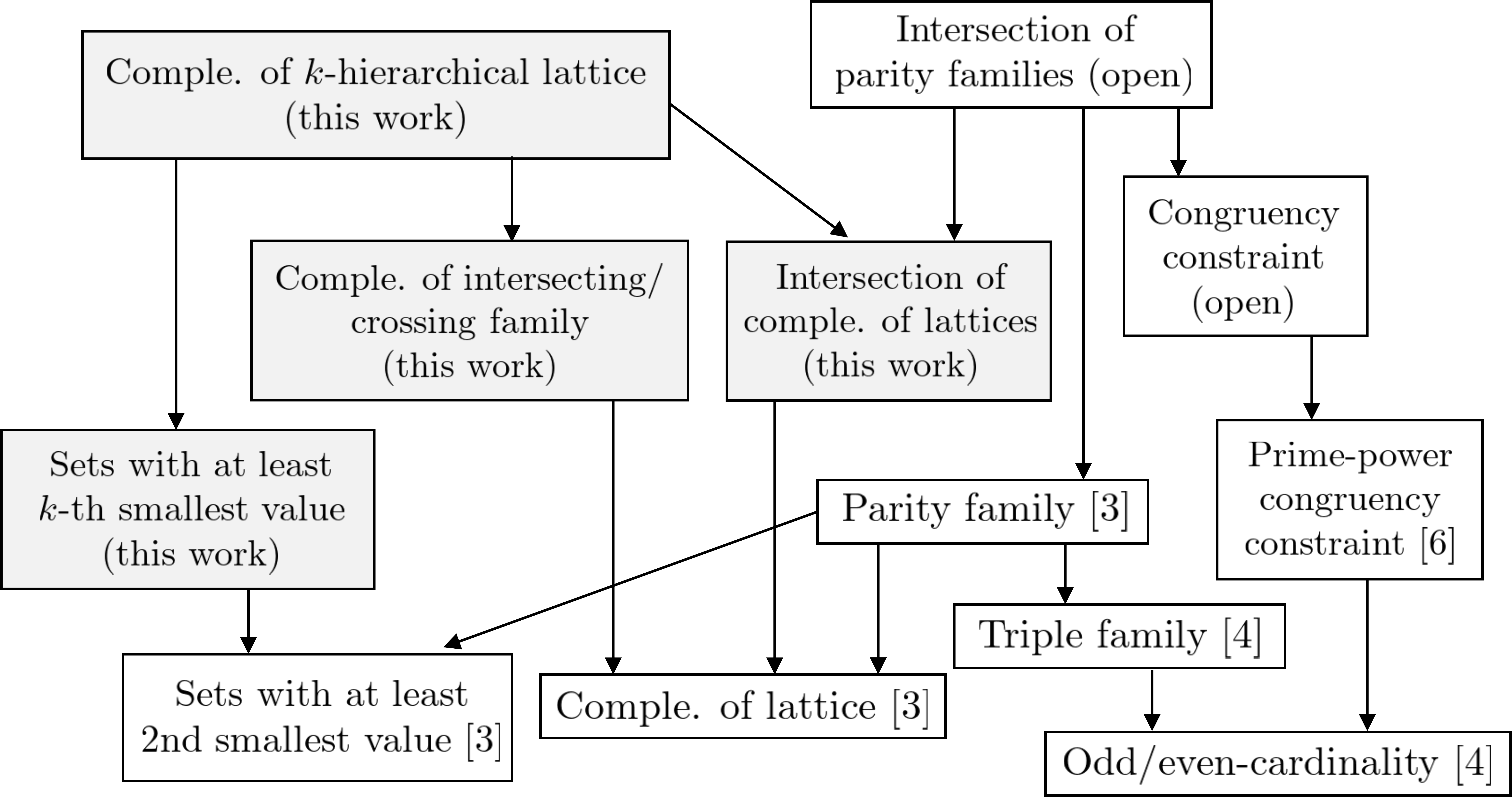}
  \caption{The inclusion relationship between various families mentioned in this paper. Arrow from Family A to Family B implies that Family A includes Family B as a special case. Polynomial-solvability of SFM over these families is shown in the references, or in this work, or still open.}
  \label{fig1}
\end{figure}

Since SFM can be solved in polynomial-time \cite{grotschel1981,iwata2001,schrijver2000}, $\tau(n)$ in Theorem \ref{thm:algorithm} is polynomial in $n$.
Hence, for a fixed positive integer $k$, SFM over the complement of a $k$-hierarchical lattice can be solved in polynomial-time by Theorem \ref{thm:algorithm}.

Since $k$-hierarchical lattices include several families as shown in Examples \ref{ex1}--\ref{ex4}, Theorem~\ref{thm:algorithm} implies polynomial-solvability of SFM over the complement of these families, such as the complement of an intersecting family, the complement of a crossing family, the intersection of the complements of lattices, and the family of sets whose function values are at least the $k$-th smallest value of a submodular function.
See Figure \ref{fig1} for the relationship between these families.

Our approach is similar to that by Goemans and Ramakrishnan \cite{goemans1995} for parity families and by N\"{a}gele, Sudakov, and Zenklusen \cite{nagele2019} for congruency constraints.
Indeed, they showed the existence of disjoint sets $S$ and $T$ of constant sizes such that a (minimal) minimizer of a submodular function over $[S,T]$ is also a minimizer over a parity family or under the prime-power congruency constraint.
This existence implies that SFM over a parity family or under the prime-power congruency constraint can be solved in polynomial-time via SFM over $[S,T]$.
We also adopt this approach for SFM over the complement of a $k$-hierarchical lattice; we show the existence of such $S$ and $T$ in Theorems \ref{thm:minimizer} and \ref{thm:minimal_minimizer}, and prove polynomial-solvability in Theorem \ref{thm:algorithm} via SFM over $[S,T]$. 

\section{Proofs}
\label{sec:proof}
In this section, we provide proofs of Theorems \ref{thm:minimizer}, \ref{thm:minimal_minimizer}, and \ref{thm:algorithm}.
Our proof strategies are similar to those presented in \cite{goemans1995,nagele2019} for parity families and congruency constraints.
We first show that for a (minimal) minimizer $X^*$ of $f$ over $\mathcal{F}$, there exist subsets $S\subseteq X^*$ and $T\subseteq V\setminus X^*$ of sizes at most $k$ such that all sets $Y$ with $S\subseteq Y\subseteq X^*$ or $X^*\subseteq Y\subseteq V\setminus T$ are in $\mathcal{F}$.
Using this fact, we show that $X^*$ is a minimizer (or the unique minimal minimizer) of $f$ over the lattice $[S,T]$.

Let $f:2^V\to \mathbb{Z}$ be a submodular function.
For a positive integer $k$, let $\mathcal{F}\subseteq 2^V$ be the complement of a $k$-hierarchical lattice.
Let $[k]:=\{1,2,\ldots,k\}$.
The following lemma plays a key role in proving our main theorems.

\begin{lemma}
\label{lem:key}
For any $X\in \mathcal{F}$, there exists $S\subseteq X$ such that $|S|\le k$ and $Y\in \mathcal{F}$ for all $Y\subseteq X$ with $S\subseteq Y$.
\end{lemma}
\begin{proof}
Take any $X\in \mathcal{F}$.
If $|X|\le k$, then setting $S=X$ satisfies the condition, and we are done.
Consider the case when $|X|\ge k+1$.
Since $2^V\setminus \mathcal{F}$ is a $k$-hierarchical lattice, there exist families $\emptyset=\mathcal{H}_0\subseteq \mathcal{H}_1\subseteq  \cdots \subseteq \mathcal{H}_k=2^V\setminus \mathcal{F}$ satisfying the condition $(*)$.
We prove the following claim by induction:
\begin{claim}
\label{clm:induction}
There exist $k$ distinct elements $v_1,v_2,\ldots,v_k\in X$ such that the following condition holds: for each $i\in [k]$, $\{v_1,v_2,\ldots,v_i\}\not\subseteq Y$ for all $Y\in \mathcal{H}_i\setminus \mathcal{H}_{i-1}$ with $Y\subseteq X$.
\end{claim}
\begin{proof}[Proof of Claim \ref{clm:induction}]
We inductively construct $v_1,v_2,\ldots,v_k$ satisfying the condition.
We first show the existence of $v_1\in X$ satisfying the condition.
Suppose to the contrary that for each $v\in X$, there exists $Y\in \mathcal{H}_1$ such that $Y\subseteq X$ and $v\in Y$.
Then, we have
\begin{align*}
\bigcup_{Y\in \mathcal{H}_1,Y\subseteq X}Y=X.
\end{align*}
Combined with the fact that $\mathcal{H}_1$ is a lattice, this implies $X\in \mathcal{H}_1$, which contradicts $X\in \mathcal{F}$.

We next show the existence of $v_2,\ldots,v_k\in X$ satisfying the condition.
As induction hypothesis, assume that there exist $v_1,v_2,\ldots,v_{j-1}$ satisfying the condition for some $2\le j\le k$.
For convenience, let $U_{j-1}:=\{v_1,v_2,\ldots,v_{j-1}\}$.
We show the existence of $v_j\in X$ satisfying the condition.
Suppose to the contrary that for each $v\in X\setminus U_{j-1}$, there exists $Y\in \mathcal{H}_j\setminus \mathcal{H}_{j-1}$ such that $U_{j-1}\cup \{v\}\subseteq Y\subseteq X$.
Define a subfamily $\mathcal{H}_j'$ of $\mathcal{H}_j\setminus \mathcal{H}_{j-1}$ as follows:
\begin{align*}
\mathcal{H}_j':=\{Y\in \mathcal{H}_j\setminus \mathcal{H}_{j-1}\mid U_{j-1}\subsetneq Y\subseteq X\}.
\end{align*}
By the assumption, for each $v\in X\setminus U_{j-1}$, there exists $Y\in \mathcal{H}_j'$ with $v\in Y$.
This implies
\begin{align*}
\bigcup_{Y\in \mathcal{H}_j'}Y=X.
\end{align*}
If $Y_1\cup Y_2\in \mathcal{H}_j\setminus \mathcal{H}_{j-1}$ for each pair of $Y_1,Y_2\in \mathcal{H}_j'$, then by the above equation we have $X\in \mathcal{H}_j\setminus \mathcal{H}_{j-1}$, a contradiction.
Hence, there exist $Y_1,Y_2\in \mathcal{H}_j'$ such that $Y_1\cup Y_2\notin \mathcal{H}_j\setminus \mathcal{H}_{j-1}$.
Since $\mathcal{H}_j\setminus \mathcal{H}_{j-1}$ satisfies the condition $(*)$, at least one of $Y_1\cup Y_2$ and $Y_1\cap Y_2$ belongs to $\mathcal{H}_{j-1}$.
Since $Y_1,Y_2\in \mathcal{H}_j'$, we have $U_{j-1}\subseteq Y_1\cap Y_2\subseteq Y_1\cup Y_2\subseteq X$, which implies that there exists $Z\in \mathcal{H}_{j-1}$ with $U_{j-1}\subseteq Z\subseteq X$.
This contradicts the induction hypothesis.
\end{proof}

Take distinct $k$ elements $v_1,v_2,\ldots,v_k\in X$ that satisfy the condition in Claim \ref{clm:induction}.
Then, we have $Y\notin \mathcal{H}_k$ for all $Y\subseteq X$ with $U_k:=\{v_1,v_2,\ldots,v_k\}\subseteq Y$.
This implies that $Y\in \mathcal{F}$ for all $U_k\subseteq Y\subseteq X$.
Hence, setting $S=U_k$ satisfies the condition in Lemma \ref{lem:key}.
\end{proof}

Similar as in the proofs in \cite{goemans1995,nagele2019}, the same argument works for the family consisting of the complements of all sets in $\mathcal{F}$ as follows.
Since $2^V\setminus \mathcal{F}$ is a $k$-hierarchical lattice, the family $\{V\setminus X\mid X\in 2^V\setminus \mathcal{F}\}$ is also a $k$-hierarchical lattice.
Hence, the family $\{V\setminus X\mid X\in \mathcal{F}\}$ is the complement of a $k$-hierarchical lattice.
Therefore, applying Lemma \ref{lem:key} to the family $\{V\setminus X\mid X\in \mathcal{F}\}$, we have the following:

\begin{lemma}
\label{lem:complement}
For any $X\in \mathcal{F}$, there exists $T\subseteq V\setminus X$ such that $|T|\le k$ and $Y\in \mathcal{F}$ for all $Y\supseteq X$ with $Y\subseteq V\setminus T$.
\end{lemma}

We also prepare the following lemma.

\begin{lemma}[cf. Theorem 7.2 in Fujishige \cite{fujishige2005}]
\label{lem:fuji72}
A set $X^*\subseteq V$ is a minimizer of a submodular function $f:2^V\to \mathbb{Z}$ if and only if $X^*$ is a minimizer of $f$ over $\{Y\subseteq V\mid Y\subseteq X^*\mathrm{\ or\ }Y\supseteq X^*\}$.
\end{lemma}

We are now ready to prove Theorems \ref{thm:minimizer} and \ref{thm:minimal_minimizer}.
\begin{proof}[Proof of Theorem \ref{thm:minimizer}]
Take a minimizer $X^*\in \mathcal{F}$ of $f$ over $\mathcal{F}$.
By Lemmas \ref{lem:key} and \ref{lem:complement}, there exist $S,T\subseteq V$ such that
\begin{itemize}
    \item $S\subseteq X^*$ and $T\subseteq V\setminus X^*$,
    \item $\max\{|S|,|T|\}\le k$,
    \item $Y\in \mathcal{F}$ for all $Y\subseteq V$ with $S\subseteq Y\subseteq X^*$ or $X^*\subseteq Y\subseteq V\setminus T$.
\end{itemize}
By the third condition, $X^*$ is a minimizer of $f$ over $\{Y\in [S,T]\mid Y\subseteq X^*\ \mathrm{or\ }Y\supseteq X^*\}$.
Hence, by Lemma \ref{lem:fuji72}, $X^*$ is a minimizer of $f$ over $[S,T]$, as desired.
\end{proof}

\begin{proof}[Proof of Theorem \ref{thm:minimal_minimizer}]
Take a minimal minimizer $X^*\in \mathcal{F}$ of $f$ over $\mathcal{F}$.
By Lemmas \ref{lem:key} and \ref{lem:complement}, there exist $S,T\subseteq V$ such that
\begin{itemize}
    \item $S\subseteq X^*$ and $T\subseteq V\setminus X^*$,
    \item $\max\{|S|,|T|\}\le k$,
    \item $Y\in \mathcal{F}$ for all $Y\subseteq V$ with $S\subseteq Y\subseteq X^*$ or $X^*\subseteq Y\subseteq V\setminus T$.
\end{itemize}
Since $S\cap T=\emptyset$, it suffices to show that $X^*$ is the unique minimal minimizer of $f$ over the lattice $[S,T]$.
By Theorem \ref{thm:minimizer}, $X^*$ is a minimizer of $f$ over $[S,T]$.
We show that $X^*\subseteq Z$ holds for every minimizer $Z\in [S,T]$ of $f$ over $[S,T]$.
By the proof of Theorem \ref{thm:minimizer}, we have $X^*\cap Z\in \mathcal{F}$ and $f(X^*\cap Z)=f(X^*)$ for every minimizer $Z$ of $f$ over $[S,T]$.
Hence, if $X^*\not\subseteq Z$ for some minimizer $Z$ of $f$ over $[S,T]$, then we have $X^*\cap Z\subsetneq X^*$, which contradicts that $X^*$ is a minimal minimizer of $f$ over $\mathcal{F}$.
Thus, we have $X^*\subseteq Z$ for every minimizer $Z$ of $f$ over $[S,T]$, as desired.
\end{proof}

Theorem \ref{thm:algorithm} follows from Theorem \ref{thm:minimal_minimizer} as follows.
The following proof is essentially the same as that in \cite{nagele2019}; we include it for the reader’s convenience.
\begin{proof}[Proof of Theorem \ref{thm:algorithm}]
Define $g:2^V\to \mathbb{Z}$ as $g(X):=(n+1)f(X)+|X|$ for every $X\subseteq V$.
Note that since $f$ is submodular, $g$ is also submodular.
For every pair of sets $S,T\subseteq V$ with $S\cap T=\emptyset$ and $\max\{|S|,|T|\}\le k$, compute a minimizer of $g$ over $[S,T]$ by polynomial-time algorithms for submodular function minimization \cite{grotschel1981,iwata2001,schrijver2000}.
This procedure can be done in $O(n^{2k}\tau (n))$ time.
Note that since $f$ is integer-valued and minimizers of $f$ are closed under intersection, $g$ has the unique minimizer over $[S,T]$, which is the unique minimal minimizer of $f$ over $[S,T]$.
Hence, by Theorem~\ref{thm:minimal_minimizer}, there exist $S,T\subseteq V$ with $S\cap T=\emptyset$ and $\max\{|S|,|T|\}\le k$ such that the unique minimizer of $g$ over $[S,T]$ minimizes $f$ over $\mathcal{F}$.
Therefore, by checking whether the unique minimizer of $g$ over $[S,T]$ belongs to $\mathcal{F}$ for each $S,T\subseteq V$ with $S\cap T=\emptyset$ and $\max\{|S|,|T|\}\le k$, and comparing the function values of such minimizers belonging to $\mathcal{F}$, we can obtain a minimizer of $f$ over $\mathcal{F}$ in $O(n^{2k}(\tau (n)+\gamma))$ time.
\end{proof}

\begin{remark}
Fujishige and Mizutani \cite{fujishige2026} showed that Lemma \ref{lem:fuji72} can be extended to more general classes of set functions called {\bf (Q1)}-submodular functions and {\bf (Q2)}-submodular functions.
A set function $f:2^V\to \mathbb{Z}$ is called {\bf (Q1)}-\textit{submodular} if for any $X,Y\subseteq V$, we have $f(X)>f(X\cap Y)$ or $f(X\cup Y)\le f(Y)$, and $f$ is called {\bf (Q2)}-\textit{submodular} if for any $X,Y\subseteq V$, we have $f(X)\ge f(X\cap Y)$ or $f(X\cup Y)< f(Y)$.
Since the proof of Theorem \ref{thm:minimizer} uses the submodularity of $f$ only in Lemma \ref{lem:fuji72}, the extensions of Lemma \ref{lem:fuji72} imply that Theorem \ref{thm:minimizer} can be extended to the classes of {\bf (Q1)}- and {\bf (Q2)}-submodular functions.
Theorem \ref{thm:minimal_minimizer} can also be extended to the class of {\bf (Q2)}-submodular functions, and the maximal version of Theorem \ref{thm:minimal_minimizer} (obtained by replacing ``minimal'' with ``maximal'') can be extended to the class of {\bf (Q1)}-submodular functions.
Note that since no polynomial-time algorithm for minimizing {\bf (Q1)}- and {\bf (Q2)}-submodular functions is known, Theorem \ref{thm:minimal_minimizer} or its maximal version for {\bf (Q1)}- and {\bf (Q2)}-submodular functions does not imply polynomial-solvability of the minimization of such functions over the complement of a $k$-hierarchical lattice.
\end{remark}

\section{Applications}
\label{sec:app}
In this section, we provide applications of Theorem \ref{thm:algorithm}.
As shown in Section \ref{sec:intro}, $k$-hierarchical lattices include several families such as an intersecting family, a crossing family, the union of $k$ lattices, and the family of subsets whose function values are at most the  $k$-th smallest value of a submodular function.
Theorem \ref{thm:algorithm} implies that SFM over the complement of such families can be solved in polynomial-time, which we will discuss below.

We first show that Theorem \ref{thm:algorithm} implies a polynomial-time algorithm for SFM over the complement of an intersecting family or a crossing family.
For an intersecting family $\mathcal{F}\subseteq 2^V$, suppose that we are given a membership oracle for $\mathcal{F}$.
Since $\mathcal{F}\cup \{\emptyset\}$ is a 2-hierarchical lattice, we can find a minimizer of a submodular function $f$ over $2^V\setminus (\mathcal{F}\cup \{\emptyset\})$ in polynomial-time.
Hence, if $\emptyset\notin \mathcal{F}$, by comparing the minimum value of $f$ over $2^V\setminus (\mathcal{F}\cup \{\emptyset\})$ with $f(\emptyset)$, we obtain a minimizer over $2^V\setminus \mathcal{F}$; otherwise, a minimizer over $2^V\setminus (\mathcal{F}\cup \{\emptyset\})$ is already a minimizer over $2^V\setminus \mathcal{F}$.
Similarly, for a crossing family $\mathcal{F}$, since $\mathcal{F}\cup \{\emptyset, V\}$ is a 2-hierarchical lattice, we can find a minimizer of $f$ over $2^V\setminus (\mathcal{F}\cup \{\emptyset, V\})$ in polynomial-time if we are given a membership oracle for $\mathcal{F}$.
Hence, by comparing the minimum value with $f(\emptyset)$ and $f(V)$ as needed, we obtain a minimizer over $2^V\setminus \mathcal{F}$.
Therefore, we have the following:

\begin{corollary}
\label{cor:intersecting_crossing}
Let $f:2^V\to \mathbb{Z}$ be a submodular function.
Let $\mathcal{F}\subseteq 2^V$ be an intersecting family or a crossing family.
Suppose that we are given a membership oracle for $\mathcal{F}$.
Then, a minimizer of $f$ over $2^V\setminus \mathcal{F}$ can be obtained in polynomial-time.
\end{corollary}

A family $\mathcal{F}\subseteq 2^V$ is called a \textit{parity family} if for all $X,Y\in 2^V\setminus \mathcal{F}$, we have
\begin{align*}
(X\cup Y\in \mathcal{F}) \iff (X\cap Y\in \mathcal{F}).
\end{align*}
An important example of a parity family is the complement of a lattice.
Goemans and Ramakrishnan \cite{goemans1995} showed that SFM over a parity family can be solved in polynomial-time; in particular, this implies polynomial-solvability of SFM over the complement of a lattice.
Since lattices are a subclass of intersecting and crossing families, Corollary \ref{cor:intersecting_crossing} can be viewed as a generalization of this tractability result for the complement of a lattice.
See Figure \ref{fig1} for the relationship between these families.
Note that parity families are not included in the class of $k$-hierarchical lattices; thus, Theorem \ref{thm:algorithm} is not a strict generalization of the result by Goemans and Ramakrishnan for parity families.

As a common generalization of SFM over parity families and congruency constraints, N\"{a}gele, Sudakov, and Zenklusen \cite{nagele2019} posed the question of whether SFM over the intersection of a constant number of parity families can be solved efficiently.
Theorem \ref{thm:algorithm} settles this question in the case when each parity family in the intersection is the complement of a lattice.
Indeed, the intersection of the complements of $k$ lattices is the complement of the union of these $k$ lattices; moreover, this union forms a $k$-hierarchical lattice.
Hence, Theorem \ref{thm:algorithm} implies the following:

\begin{corollary}
\label{cor:parity_intersection}
Let $f:2^V\to \mathbb{Z}$ be a submodular function.
For a constant positive integer $k$, let $\mathcal{F}\subseteq 2^V$ be the intersection of the complements of $k$ lattices.
Suppose that we are given a membership oracle for $\mathcal{F}$.
Then, a minimizer of $f$ over $\mathcal{F}$ can be obtained in polynomial-time. 
\end{corollary}

See Figure \ref{fig1} for the relationship between the families mentioned above.
While $k$-hierarchical lattices and prime-power congruency constrained set families are incomparable under inclusion, both our algorithm and the algorithm of N\"{a}gele, Sudakov, and Zenklusen rely on the existence of disjoint sets $S,T$ of constant sizes such that a (minimal) minimizer of a submodular function over $[S,T]$ is also a minimizer over the complement of a $k$-hierarchical lattice or under the prime-power congruency constraint.
Since no intersection of parity families is known that fails to admit such $S$ and $T$, extending this property to the intersection of parity families is a possible direction to derive a polynomial-time algorithm for SFM over that class.

As a corollary of polynomial-solvability of SFM over parity families, Goemans and Ramakrishnan \cite{goemans1995} showed that the second smallest value of a submodular function can be obtained in polynomial-time.
For a fixed positive integer $k$, Vazirani and Yannakakis \cite{vazirani1992} showed that the $k$-th smallest value of $s$-$t$ cuts of a directed or undirected graph can be computed in polynomial-time.
They showed that there exist disjoint vertex sets $S,T$ of sizes at most $k-1$ such that a minimum cut separating $S\cup \{s\}$ and $T\cup \{t\}$ attains the $k$-th smallest value of $s$-$t$ cuts, using the representation of all the minimum $s$-$t$ cuts via a partial order among the strongly connected components in the residual graph.
Since the family of subsets whose function values are at most the $k$-th smallest value of a submodular function forms a $k$-hierarchical lattice, Theorem~\ref{thm:algorithm} implies a generalization of the result by Vazirani and Yannakakis to the setting of submodular functions:

\begin{corollary}
Let $f:2^V\to \mathbb{Z}$ be a submodular function.
For a constant positive integer $k$, a set attaining the $k$-th smallest value of $f$ can be obtained in polynomial-time.
\end{corollary}
\begin{proof}
We prove by induction on $k$.
The case of $k=1$ is done by polynomial-time algorithms for SFM.
Consider the case when $k\ge 2$.
Let $\mathcal{F}\subseteq 2^V$ be the family of subsets whose function values are at most the $(k-1)$-th smallest value of $f$.
Since $\mathcal{F}$ is a $(k-1)$-hierarchical lattice, Theorem \ref{thm:algorithm} implies that SFM over $2^V\setminus \mathcal{F}$ can be solved in polynomial-time if we can determine in polynomial-time whether any given $X\subseteq V$ belongs to $\mathcal{F}$ or not.
Hence, it suffices to show whether any given $X\subseteq V$ belongs to $\mathcal{F}$ can be determined in polynomial-time.
By induction hypothesis, a set attaining the $i$-th smallest value of $f$ can be obtained in polynomial-time for each $i=1,2,\ldots,k-1$.
Hence, for any given $X\subseteq V$, whether $X\in \mathcal{F}$ can be determined in polynomial-time.
\end{proof}

\section*{Acknowledgements}
We are grateful to Satoru Fujishige for valuable comments and suggestions.
This work was supported by JST ERATO Grant Number JPMJER2301, Japan.

\bibliography{main}
\bibliographystyle{plain}

\end{document}